\documentclass[12pt]{amsart}
\usepackage{amscd,amssymb,hyperref}
\usepackage[PostScript=dvips]{diagrams}

\newtheorem{thm}{Theorem}[section]
\newtheorem{lem}[thm]{Lemma}
\newtheorem{cor}[thm]{Corollary}
\newtheorem{prop}[thm]{Proposition}

\def\square{\vbox{
      \hrule height 0.4pt
      \hbox{\vrule width 0.4pt height 5.5pt \kern 5.5pt \vrule width 0.4pt}
      \hrule height 0.4pt}}

\def\id{\mathrm{id}}

\def\Ker{\mathrm{K er}}
\def\ch\mathrm{c h}

\def\RP{\mathbb{R}\mathrm{P}}

\long\def\symbolfootnote[#1]#2{\begingroup%
\def\thefootnote{\fnsymbol{footnote}}\footnote[#1]{#2}\endgroup}

\newcommand{\Z}{\mathbb{Z}}

\newcommand{\Brun}{\mathrm{Brun}}

\numberwithin{equation}{section}

\newcommand{\auths}[1]{\textrm{#1},}
\newcommand{\artTitle}[1]{\textsl{#1},}
\newcommand{\jTitle}[1]{\textrm{#1}}
\newcommand{\Vol}[1]{\textbf{#1}}
\newcommand{\Year}[1]{\textrm{(#1)}}
\newcommand{\Pages}[1]{\textrm{#1}}

\title{On Cohen  braids }

\author{V. G. Bardakov}
\address{Sobolev Institute of Mathematics, Novosibirsk 630090, Russia}
\email{bardakov@math.nsc.ru}

\author{V. V. Vershinin}
\address{D\'epartement des Sciences Math\'ematiques,
                               Universit\'e Montpellier II,
Place Eug\`ene Bataillon,
34095 Montpellier cedex 5, France} \email{ vershini@math.univ-montp2.fr}
\address{Sobolev Institute of Mathematics, Novosibirsk 630090,
Russia }
\email{ versh@math.nsc.ru}

\author{J. Wu}
\address{Department of Mathematics, National University of Singapore, 2 Science Drive 2
Singapore 117542} \email{matwuj@nus.edu.sg}
\urladdr{www.math.nus.edu.sg/\~{}matwujie}


\begin{document}

\begin{abstract}
For a general surface $M$ and  an arbitrary braid $\alpha$ from the surface braid group $B_{n-1}(M)$
we study the system of equations $
d_1\beta=\cdots=d_{n}\beta=\alpha,
$ where the operation $d_i$ is deleting of $i$-th strand.             
We obtain that if $M\not=S^2$ or $\RP^2$ this system of equations has a solution 
$\beta\in B_{n}(M)$ if and only if
$
d_1\alpha=\ldots=d_n\alpha.
$
The set of braids satisfying the last system of equations we call {\it Cohen} braids. We also construct a set of generators for the groups of Cohen braids.
In the cases of the sphere and the projective plane we give some examples for the  small number of strands.
\end{abstract}

\maketitle

\tableofcontents

\section{Introduction}
Let $M$ be a general 
connected surface, possibly with boundary components (we can consider $M$ as a compact surface with a finite number of punctures). We denote by $B_n(M)$ the $n$-strand braid group on a surface $M$. 
In the work \cite{BMVW} \textit{Brunnian} braids on the surface $M$ were studied.
Brunnian
means a braid that becomes trivial after
removing any one of its strands.
 See the formal definition  in \cite{BMVW}. Here we mention that it is done with the help of the operations
$$d_i\colon B_n(M)\to B_{n-1}(M)$$
which 
 are obtained (roughly speaking) by forgetting the $i$-th strand, $1\leq i\leq n$.
We can interpret a Brunnian  braid $\beta\in B_n(M)$ as a solution of the system of $n$
equations 
\begin{equation*}
\begin{cases}
 d_1(\beta)=1, \\
\dots\\
 d_n(\beta)=1.
 \end{cases}
\end{equation*}
Let $\Brun_n(M)$ denote the set of the $n$-strand Brunnian braids. Then $\Brun_n(M)$ forms 
a subgroup of $B_n(M)$.

In the present paper we replace the unit element of the group  by an arbitrary braid $\alpha\in B_{n-1}(M)$ and we  
ask the following question: does there exist a braid $\beta\in B_n(M)$ such that it is a solution of the following system of equations
\begin{equation}
\begin{cases}
 d_1(\beta)=\alpha, \\
\dots\\
 d_n(\beta)=\alpha.
 \end{cases}\label{eq:sys2}
\end{equation}   
Apart from Brunnian braids the following example can be given. Let $\alpha$ be the Garside element
$\Delta_{n-1}\in B_{n-1}(M)$. Then $\Delta_n\in B_n(M)$ is a solution of system (\ref{eq:sys2}).

The main result of the paper is the following.
\begin{thm}\label{main}
Let $M$ be any connected $2$-manifold such that $M\not=S^2$ or $\RP^2$ and let $\alpha\in B_{n-1}(M)$. Then the system of equations (\ref{eq:sys2})
for $n$-strand braids $\beta$ has a solution if and only if $\alpha$ satisfies the condition that
$$
d_1\alpha=\ldots=d_{n-1}\alpha.
$$
\end{thm}
The technique of the proof is based on the bi-$\Delta$-group structure on the pure braid groups of connected $2$-manifolds with nonempty boundary~\cite{Wu4} as well as the determination of Brunnian braids on general connected $2$-manifolds~\cite{BMVW}. Our technology could not apply for the exceptional cases $M=S^2$ or $\RP^2$, and so it remains open whether system~(\ref{eq:sys2}) has a solution for any $\alpha$ satisfying the condition in the theorem.

The paper is organized as follows. In section~2 we remind bi-$\Delta$-structures,
 define Cohen braids and prove the main theorem. In section 3 we study Cohen braids of
  the disc with  small number of strands and construct a generating set for Cohen braids of the disc for an arbitrary number of strands. In section~4 we study Cohen braids of the sphere and
of the projective plane and a map from the group of Cohen braids of the disc to the
group of Cohen braids of the sphere. In section~5 we propose generalizations of Cohen
braids   and in section~6 we discuss some open questions.
     
\section{Bi-$\Delta$-structure, Cohen groups and Brunnian braids}
Let $M$ be a connected manifold with $\partial M\not=\emptyset$. Let $a$ be a point in a collar of $\partial M$
$$
\partial M\times [0,1)\subseteq M.
$$
Then the map
$$
\begin{array}{c}
F(M,n)\simeq F(M\smallsetminus \partial M\times [0,1),n)\rTo F(M,n+1),\\
(z_1,\ldots,z_n)\mapsto (z_1,\ldots,z_{i-1},a,z_{i+1},\ldots,z_n)\\
\end{array}
$$
induces a group homomorphism
$$
d^i\colon B_n(M)=\pi_1(F(M,n)/\Sigma_n)\rTo B_{n+1}(M)=\pi_1(F(M,n+1)/\Sigma_{n+1})
$$
for $1\leq i\leq n+1$. Intuitively, $d^i$ is given by adding a trivial strand in position $i$. According to~\cite[Example 1.2.8]{Wu4}, the sequence of groups $\{B_{n+1}(M)\}_{n\geq 0}$ with faces relabeled as $\{d_0,d_1,\ldots\}$ and cofaces relabeled as $\{d^0,d^1,\ldots\}$ forms a bi-$\Delta$-set structure. Namely the following identities hold:
\begin{enumerate}
\item $d_jd_i=d_id_{j+1}$ for $j\geq i$;
\item $d^jd^i=d^{i+1}d^j$ for $j\leq i$;
\item $d_jd^i=\left\{
\begin{array}{lcl}
d^{i-1}d_j&\textrm{ if }& j<i,\\
\id&\textrm{ if }& j=i,\\
d^id_{j-1}&\textrm{ if }& j>i.\\
\end{array}
\right.$
\end{enumerate}
Let us  restrict to the pure braid groups, then the sequence of these  groups $\{P_{n+1}(M)\}_{n\geq0}$ is a bi-$\Delta$-group. According to~\cite[Proposition 1.2.9]{Wu4}, we have the following decomposition.
\begin{prop}\label{proposition6.1}
Let $M$ be a connected $2$-manifold with nonempty boundary. Then $P_n(M)$ is the (iterated) semi-direct product of the subgroups
$$
d^{i_k}d^{i_{k-1}}\cdots d^{i_1}(\Brun_{n-k}(M)),
$$
$1\leq i_1<i_2<\cdots<i_k\leq n$, $0\leq k\leq n-1$, with lexicographical
order from the right.\hfill $\Box$
\end{prop}
The braids in the subgroup $d^{i_k}d^{i_{k-1}}\cdots d^{i_i}(\Brun_{n-k}(M))$ can be described as $(n-k)$-strand Brunnian braids with $k$ dots (or straight lines) in positions $i_1, \ldots, i_k$. This 
is connected with the inverse braid monoids~\cite{Vershinin}.

We return to the case when $M$ is an arbitrary connected $2$-manifold. Let us define 
a set
$$
\mathfrak{H}^B_n(M)=\{\beta\in B_n(M)\ | \ d_1\beta=d_2\beta=\cdots=d_n\beta\}.
$$
Namely, $\mathfrak{H}^B_n(M)$ consists of $n$-strand braids such that it stays the same braid after removing any one of its strands. For the general $\Delta$-group it  was called {\it Cohen set} in \cite{Wu4}. This is based on Fred Cohen constructions in homotopy theory \cite{C1, C2}.
So, we propose to call this {\it Cohen braids}. We denote Cohen braids for the disc simply by
$
\mathfrak{H}^B_n$ as well as Brunnian braids of the disc by $\Brun_n$. 

A typical element in $\mathfrak{H}^B_n(M)$ is the half-twist braid, which can be expressed, for example, by the formula
$$
\Delta_n=(\sigma_1\sigma_2\cdots\sigma_{n-1})(\sigma_1\sigma_2\cdots\sigma_{n-2})\cdots (\sigma_1\sigma_2)\sigma_1.
$$
\begin{prop}\label{proposition6.2}
Let $M$ be any connected $2$-manifold. Then the set
$\mathfrak{H}^B_n(M)$ is a subgroup of $B_n(M)$. Moreover
$d_i(\mathfrak{H}^B_n(M))\subseteq \mathfrak{H}^B_{n-1}(M)$ and the map
$$
d_1=d_2=\cdots=d_n\colon \mathfrak{H}_n^B(M)\to \mathfrak{H}_{n-1}^B(M)
$$
is a group homomorphism.
\end{prop}
\begin{proof}
Let $\beta,\gamma\in \mathfrak{H}^B_n(M)$. Then
\begin{equation}\label{equation6.2}
d_i(\beta\gamma)=d_i(\beta) d_{i\cdot\beta}(\gamma)=d_1(\beta)d_1(\gamma)
=d_i(\beta)d_i(\gamma)
\end{equation}
for $1\leq i\leq n$. Thus $\beta\gamma\in \mathfrak{H}^B_n(M)$. From
$$
1=d_i(\beta^{-1}\beta)=d_i(\beta^{-1})d_{i\cdot\beta^{-1}}(\beta)=d_i(\beta^{-1})d_1(\beta),
$$
we have
$$
d_i(\beta^{-1})=(d_1(\beta))^{-1}
$$
for $1\leq i\leq n$. Thus $\mathfrak{H}^B_n(M)$ is a subgroup of $B_n(M)$.

As $\beta\in \mathfrak{H}^B_n(M)$,  the identities
$$
d_j(d_i\beta)=d_j(d_1\beta)=d_1(d_{j+1}\beta)=d_1(d_1\beta)
$$
give: $d_i\beta=d_1\beta\in \mathfrak{H}^B_{n-1}(M)$. From equation~(\ref{equation6.2}),
$$
d_1=d_i\colon \mathfrak{H}^B_n(M)\to \mathfrak{H}^B_{n-1}(M)
$$
is a group homomorphism and hence the result.
\end{proof}

\begin{prop}\label{proposition6.3}
Let $M$ be any connected $2$-manifold. Let $n\geq 2$. Then $\mathfrak{H}^B_n(M)\cap P_n(M)$ is a subgroup of $\mathfrak{H}^B_n(M)$ of index $2$.
\end{prop}
\begin{proof}
Consider the short exact sequence of groups
$$
P_n(M)\rInto B_n(M)\rOnto^{q_n}\Sigma_n.
$$
The face function $d_i\colon B_n(M)\to B_{n-1}(M)$ induces a unique face $$d_i^{\Sigma}\colon \Sigma_n\to \Sigma_{n-1}$$ such that
$$
d_i^{\Sigma}\circ q_n=q_{n-1}\circ d_i^B
$$
for each $1\leq i\leq n$. Let $\mathfrak{H}_n^{\Sigma}=q_n(\mathfrak{H}_n^B(M))$. Then there is a left exact sequence
$$
1\rTo \bigcap_{i=1}^n\Ker(d_i^{\Sigma})\rTo \mathfrak{H}^{\Sigma}_n\rTo^{d_1^{\Sigma}}\mathfrak{H}^{\Sigma}_{n-1}.
$$
By direct calculation, $\bigcap_{i=1}^n\Ker(d_i^{\Sigma})=\{1\}$ for $n\geq 3$. Thus
$$
d_1^{\Sigma}\circ \cdots \circ d_1^{\Sigma}\colon \mathfrak{H}^{\Sigma}_n\longrightarrow \mathfrak{H}^{\Sigma}_2=\Z/2
$$
is a monomorphism for $n\geq 3$. Since the half-twist braid $\Delta_n$ has nontrivial image in $\mathfrak{H}^{\Sigma}_n$, we have
$$
\mathfrak{H}^{\Sigma}_n=\Z/2
$$
for $n\geq 2$ and hence the result.
\end{proof}

Let $\mathfrak{H}_n(M)=\mathfrak{H}^B_n(M)\cap P_n(M)$. Then $d_1(\mathfrak{H}_n(M))\subseteq \mathfrak{H}_{n-1}(M)$. This gives a tower of groups
$$
\cdots \rTo^{d_1} \mathfrak{H}_n(M)\rTo^{d_1} \mathfrak{H}_{n-1}(M)\rTo^{d_1}\cdots.
$$
Let $\mathfrak{H}(M)=\lim\limits_n\mathfrak{H}_n(M)$ be the inverse limit of the tower of groups.

\begin{prop}\label{proposition6.4}
Let $M$ be any connected $2$-manifold such that $M\not=S^2$ or $\RP^2$. Then
$$
d_1\colon \mathfrak{H}_n(M)\to\mathfrak{H}_{n-1}(M)
$$
is an epimorphism for each $n\geq 2$.
\end{prop}
\begin{proof}
By the definition of Brunnian braids there is a  short left exact  sequence (we do not claim that $d_1$ is onto at the moment)
\begin{diagram}
\Brun_{k}( M)&\rInto &\mathfrak{H}_k( M)&\rTo^{d_1}&\mathfrak{H}_{k-1}( M).
\end{diagram}
If $M$ is a connected $2$-manifold with nonempty boundary, the assertion follows from ~\cite[Proposition 3.1]{Wu4} by using the bi-$\Delta$-structure on $\{P_{n+1}(M)\}_{n\geq0}$.

Suppose that $M$ is a closed manifold and $M\not=S^2$ or $\RP^2$. Let $\tilde M=M\smallsetminus\{q_1\}$. We show by induction that
$$
\mathfrak{H}_k(\tilde M)\rTo \mathfrak{H}_k(M)
$$
is onto for each $k$. Clearly the statement holds for $k=1$. Assume that the statement holds for $k-1$ with $k\geq 2$. Consider the commutative diagram
\begin{diagram}
\Brun_{k}(\tilde M)&\rInto &\mathfrak{H}_k(\tilde M)&\rOnto^{d_1}&\mathfrak{H}_{k-1}(\tilde M)\\
\dTo&&\dTo&&\dOnto\\
\Brun_{k}(M)&\rInto &\mathfrak{H}_k(M)&\rTo^{d_1}&\mathfrak{H}_{k-1}(M),\\
\end{diagram}
where the right column is onto by induction. By Theorem ~1.1 of \cite{BMVW}, $$\Brun_k(\tilde M)\to \Brun_k(M)$$ is onto and so the middle column $\mathfrak{H}_k(\tilde M)\to \mathfrak{H}_k(M)$ is onto. The induction is finished and so
$$
\mathfrak{H}_n(\tilde M)\to \mathfrak{H}_{n}(M)
$$
is an epimorphism for each $n$. From the right square in the above diagram,
$$
d_1\colon \mathfrak{H}_n(M)\to \mathfrak{H}_{n-1}(M)
$$
is an epimorphism for each $n\geq 2$ and hence the result.
\end{proof}

\begin{proof}[Proof of Theorem 1.1]
If there exists $\beta$ such that $d_1\beta=\cdots=d_{n}\beta=\alpha$, then $\alpha\in \mathfrak{H}^B_{n-1}(M)$ by Proposition~\ref{proposition6.2} and so $d_1\alpha=\cdots=d_{n-1}\alpha$.

Conversely suppose that $d_1\alpha=\ldots=d_{n-1}\alpha$. Then $\alpha\in \mathfrak{H}^B_{n-1}(M)$. If $\alpha\in \mathfrak{H}_{n-1}(M)$, then the equation in the statement has a solution for $\alpha$ by Proposition~\ref{proposition6.4}. If $\alpha\not\in\mathfrak{H}_{n-1}(M)$, then $\Delta_{n-1}\alpha\in \mathfrak{H}_{n-1}(M)$. Thus there exists $\gamma$ such that $d_1\gamma=\cdots=d_{n}\gamma=\Delta_{n-1}\alpha$. Since $d_1\Delta_{n}=\cdots=d_{n}\Delta_{n}=\Delta_{n-1}$, we have
$$
d_1(\Delta_{n}^{-1}\gamma)=\cdots=d_{n}(\Delta_{n}^{-1}\gamma)=\alpha
$$
and hence the result.
\end{proof}
\begin{cor}\label{corollary6.5}
Let $M$ be any connected $2$-manifold such that $M\not=S^2$ or $\RP^2$. Then there exists the 
following short exact sequence 
\begin{equation}
1\to\Brun_n(M)\to \mathfrak{H}_n(M)\buildrel d_1\over\longrightarrow \mathfrak{H}_{n-1}(M)\to 1
\label{eq:exseH}
\end{equation}
which connects the $n$th, $(n-1)$th Cohen braid groups and Brunnian braids on $n$-strands. \hfill $\Box$
\end{cor}
By definition $B_1(M) = P_1(M) = \pi_1(M,p_1) = \pi_1(M).$ Let $g$ be an element 
of $\pi_1(M,p_1).$
Then $g$ is represented by a loop with the ends in $p_1$. Under canonical inclusion
\begin{diagram}
P_{1}( M)&\rInto& P_n( M)
\end{diagram}
element $g$ defines an element in $P_n(M)$ which we denote also by $g$.
In the group $P_n(M),$ $n \geq 2$ define elements
$$
g_1 = g,~~g_2 = \sigma_1 g_1 \sigma_1^{-1},~~g_3 = \sigma_2 g_2 \sigma_2^{-1},~~\ldots,
g_n = \sigma_{n-1} g_{n-1} \sigma_{n-1}^{-1}.
$$
Any $g_i$ is represented by a braid that consists of a loop with the ends in the point $p_i$ (which is a starting point of
 the $i$-th string) while the rest of the braid is trivial. Define $h = g_1 g_2 \ldots g_n.$
\begin{prop}
Element $h$ is a Cohen braid in $P_n(M).$ \hfill  $\Box$
\end{prop}
Denote by $K_1$ a set of generators for $\pi_1(M,p_1)$ and by $K_n \subseteq P_n(M)$ the set of elements $h$
constructed from  elements of $H_1$ by the rules described above. Then $K_n \subseteq \mathfrak{H}_n(M).$
Our main conjecture on a generating set of $\mathfrak{H}_n(M)$ is as follows.

{\bf Conjecture.} Let $M$ be a connected surface, $M \not= S^2$ or $\mathbb{R}P^2.$ The group $\mathfrak{H}_n(M)$ is generated by
elements $K_n$ and the group $\mathfrak{H}_n$
of the Cohen braids on the disc.

\medskip

\section{Cohen and Brunnian braids in the braid group of the disk}
\subsection{General properties}
Recall that the generators for the pure braid group $P_n$ are given by
\begin{equation*}\label{eq:Aij}
A_{i,j}=\sigma_{j-1}\sigma_{j-2}\cdots\sigma_{i+1}\sigma_i^2\sigma_{i+1}^{-1}\cdots
\sigma_{j-2}^{-1}\sigma_{j-1}^{-1}
\end{equation*}
for $1\leq i<j\leq n$ (see \cite{Ve9}, for example). 
\begin{prop}\label{norm}
If $n \geq 3$ then the group $\mathfrak{H}_n^B$ is not normal in $B_n$. If a Cohen braid $\alpha$ in
$\mathfrak{H}_n$ is such that for any generator $A_{ij} \in P_n$ the braid
$A_{ij}^{-1} \alpha A_{ij}$ is Cohen then the element $d_1(\alpha)$ lies in the center of $P_{n-1}.$
\end{prop}
\begin{proof}
The element $\Delta_3$ lies in $\mathfrak{H}_3^{B}.$ After conjugation by  $\sigma_1$ 
it is equal to $\sigma_2 \, \sigma_1^2$:
$$
\sigma_1^{-1} \, \Delta_3 \, \sigma_1 = \sigma_2 \, \sigma_1^2.
$$
Consider the action of operations $d_i$ on this element:
$$
d_1(\sigma_2 \, \sigma_1^2) = \sigma_1,~~~d_2(\sigma_2 \, \sigma_1^2) = \sigma_1^2,
~~~d_3(\sigma_2 \, \sigma_1^2) = e,
$$
i.e. $\sigma_2 \, \sigma_1^2$ is not a Cohen braid.
Hence the group $\mathfrak{H}_3^B$ is not normal in $B_3.$

To prove the second statement let us apply the operations $d_k$ to
the element $A_{ij}^{-1} \alpha A_{ij}$:
$$
d_k (A_{ij}^{-1} \alpha A_{ij}) = d_k(A_{ij}^{-1}) \, d_k(\alpha) \, d_k(A_{ij}).
$$
The result depends  of $k$.  There are the following cases:

1) $1 \leq k < i < j \leq n;$

2) $k = i$ or $k = j;$

3) $1 \leq i < k < j \leq n;$

4)  $1 \leq j < k \leq n.$

In the first case we have
$$
d_k (A_{ij}^{-1} \alpha A_{ij}) = A_{i-1,j-1}^{-1} \, d_k(\alpha) \, A_{i-1,j-1}.
$$

In the second  case we have
$$
d_k (A_{ij}^{-1} \alpha A_{ij}) =  d_k(\alpha).
$$

In the third case we have
$$
d_k (A_{ij}^{-1} \alpha A_{ij}) = A_{i,j-1}^{-1} \, d_k(\alpha) \, A_{i,j-1}.
$$

In the forth  case we have
$$
d_k (A_{ij}^{-1} \alpha A_{ij}) = A_{i,j}^{-1} \, d_k(\alpha) \, A_{i,j}.
$$

If we denote $d_k(\alpha)$ by $\beta \in \mathfrak{H}_{n-1}$ then the following relations are true
$$
\beta = A_{i-1,j-1}^{-1} \, \beta \, A_{i-1,j-1} = A_{i,j-1}^{-1} \, \beta \, A_{i,j-1} =
A_{i,j}^{-1} \, \beta \, A_{i,j}
$$
or
$$
\beta = \beta \, [\beta, A_{i-1,j}] = \beta \, [\beta, A_{i,j-1}] = \beta \, [\beta, A_{i,j}].
$$
Hence, $\beta$ commutes with any generator of $P_{n-1},$ i.e. $\beta \in Z(P_{n-1}) = Z(B_{n-1}).$
\end{proof}
Braids $\Delta_n^k$ are Cohen.
It was proved in \cite[Lemma 1]{B2} that in $B_3$ the following formula is true
$$
\Delta_3^{2k} = A_{12}^k \, (A_{13} A_{23})^k.
$$
It is not difficult to see that in general case we have the similar result
$$
\Delta_n^{2k} = A_{12}^k \, (A_{13} A_{23})^k \ldots (A_{1n} A_{2n} \ldots A_{n-1,n})^k.
$$
 Also, it is not difficult to see that  the braids
$$
A_{12}^k \, (A_{13}^k A_{23}^k) \ldots (A_{1n}^k A_{2n}^k \ldots A_{n-1,n}^k),~~~k \in \mathbb{Z}
$$
are Cohen. 

\begin{prop}
Let
$$
\alpha = [A_{i_1,j_1}^{k_1}, A_{i_2,j_2}^{k_2}, \ldots, A_{i_l,j_l}^{k_l}],~~~k_i \in \mathbb{Z},
$$
be a commutator in $P_n$ with some arrangement of the brackets. If
$$
\{ i_1, j_1, i_2, j_2, \ldots, i_l, j_l \} = \{ 1, 2, \ldots, n \}
$$
then $\alpha \in {\rm Brun}_n.$ \hfill $\Box$
\end{prop}

\subsection{Cohen braids in $P_2$ and $P_3$} 
For braids on two strings we have: $\mathfrak{H}_2= \Brun_2=P_2=\Z$ and the exact sequence (\ref{eq:exseH})
gives
\begin{equation*}
1\to\Brun_3\to \mathfrak{H}_3\buildrel d_1\over\longrightarrow \Z\to 1.
\end{equation*}
In $\mathfrak{H}_3$ it is the central element $\Delta_3^2$ which is mapped to the generator of 
$\mathfrak{H}_2$, so $\mathfrak{H}_3\cong \Z\times\Brun_3$ and for the next step we have
\begin{equation*}
1\to\Brun_4\to \mathfrak{H}_4\longrightarrow \Z\times \Brun_3\to 1.
\end{equation*}
\smallskip
More precisely, in this case $P_3 = Z \times U_3,$ where $Z = \langle \Delta_3^2 \rangle$
is the center of $B_3$ and $P_3$, $U_3 = \langle A_{13}, A_{23} \rangle \simeq F_2.$ Any element $\beta$ in
$U_3$ has the form
\begin{equation}\label{semi}
\beta = A_{13}^k \, A_{23}^l \, \gamma,~~~k, l \in \mathbb{Z},  \gamma \in U_3'.
\end{equation}
For braids on three strings ${\rm Brun}_3 = U_3'$ (where prime denotes the commutator subgroup). Hence,  $\gamma$ is a Brunnian braid. The action of operations $d_k$ gives:
$$
d_1(\beta) = A_{12}^l,~~~d_2(\beta) = A_{12}^k,~~~d_3(\beta) = e.
$$
Therefore, the braid $\beta$ is  Cohen braid if and only if it is  Brunnian braid. Thus we have a full description
of the Cohen braids in $B_3$: any such braid has a form
$$
\Delta_3^k \, \gamma,~~~k \in \mathbb{Z},  \gamma \in U_3'
$$
and any Cohen braid in $P_3$ has a form
$$
\Delta_3^{2k} \, \gamma,~~~k \in \mathbb{Z},  \gamma \in U_3'.
$$

\subsection{Cohen braids in $P_4$} The group $P_4$ has the decomposition
$$
P_4 = \langle \Delta_4^2 \rangle \times (U_3 \leftthreetimes U_4),
$$
where $U_4=\langle A_{14}, A_{24}, A_{34} \rangle$ 
and the commutator subgroup of $P_4$ decomposes as follows
$$
P_4' = U_3' \leftthreetimes U_4'.
$$
The commutator subgroup $U_3'$ is equal to ${\rm Brun}_3,$ but as a subgroup of $P_4$ it does not lie in
${\rm Brun}_4.$

Take an element $\alpha \in P_4.$ We want to understand under what conditions the element
$\alpha$ lies in
$\mathfrak{H}_4.$ Let $\alpha = \Delta_4^{2k} \beta$ for some integer $k$ and
$\beta \in U_3 \leftthreetimes U_4.$ Since  $\Delta_4^{2k} \in \mathfrak{H}_4,$  we need to understand
under what conditions $\beta$ lies in
$\mathfrak{H}_4.$ Let
$$
\beta = A_{13}^k A_{23}^l A_{14}^m A_{24}^p A_{34}^q \gamma,~~~k, l, m, p, q \in \mathbb{Z},~~\gamma \in
(U_3 \leftthreetimes U_4)'.
$$
Then we have
$$
d_1(\beta) = A_{12}^l A_{13}^p A_{23}^q  d_1(\gamma),~~~
d_2(\beta) = A_{12}^k A_{13}^m A_{23}^q  d_2(\gamma),
$$
$$
d_3(\beta) =  A_{13}^m A_{23}^p  d_3(\gamma),~~~
d_4(\beta) =  A_{13}^k A_{23}^l  d_4(\gamma).
$$
By the definition $\beta \in \mathfrak{H}_4$ if and only if
$$
d_1(\beta) = d_2(\beta) = d_3(\beta) = d_4(\beta).
$$
Since $d_i$ is an homomorphism for pure braids $d_i(\gamma) \in P_3'$ for all $i =1, 2, 3, 4$ and hence  we have
$$
A_{12}^l A_{13}^p A_{23}^q = A_{12}^k A_{13}^m A_{23}^q = A_{13}^m A_{23}^p = A_{13}^k A_{23}^l,
$$
$$
d_1(\gamma) = d_2(\gamma) = d_3(\gamma) = d_4(\gamma).
$$
From these equalities
$$
l = k = 0,~~~p = m = k, ~~~q = p = l,
$$
so $\beta \in (U_3 \leftthreetimes U_4)'.$

We want to understand now are there  Cohen, but non-Brunnian braids in the commutator subgroup $P_4'$. Take
the commutator
$$
\gamma_4 = [A_{12}^2 \, (A_{13}^2 A_{23}^2) \,  (A_{14}^2 A_{24}^2 A_{34}^2),
A_{12}^3 \,  (A_{13}^3 A_{23}^3) \,  (A_{14}^3 A_{24}^3 A_{34}^3)] \in P_4'.
$$
Then
$$
d_1(\gamma_4) = d_2(\gamma_4) = d_3(\gamma_4) = d_4(\gamma_4) = \gamma_3 \in P_3',
$$
where
$$
\gamma_3 = [A_{12}^2 A_{13}^2 A_{23}^2, A_{12}^3 A_{13}^3 A_{23}^3].
$$
To prove that $\gamma_3 \not= e$ find its normal form. Using the conjugation rules in $P_3$ we have
$$
\gamma_3 = A_{23}^{-2} A_{13}^{-1} A_{23} A_{13} A_{23}^{-2} A_{13}^{-2} A_{23} A_{13}^{2} A_{23} A_{13}^{-1}
A_{23} A_{13}^{-1} A_{23}^{-1} A_{13}^{2} A_{23}^{3}.
$$
Hence, $\gamma_4$ is a Cohen but non-Brunnian braids in  $P_4'$.

To find the set of generators for $\mathfrak{H}_4$ we construct a lifting of elements in $\mathfrak{H}_3$ into
 $\mathfrak{H}_4$. To illustrate this let us consider the following example.

 {\bf Example.}  Let us take an element $\alpha_{lm} = [A_{13}^l, A_{23}^m] \in
  \mathfrak{H}_3$ for some non-zero integers
$l$ and $m$. If we consider $\alpha_{lm}$
 as an element in $P_4$ (under the canonical inclusion
$P_{3}\hookrightarrow P_4$),
 then $\alpha_{lm}$ is ${3}$-Brunnian (it becomes trivial after deleting the third string) but does not lie in $\mathfrak{H}_4$
since $d_4(\alpha_{lm}) = \alpha_{lm}$. Let us make the following notation:
 $$
\widetilde{\alpha}_{lm} = [A_{13}^l, A_{23}^m] [A_{24}^l, A_{34}^m] [A_{14}^l, A_{34}^m] [A_{14}^l, A_{24}^m]
\in P_4.
 $$
Then we obtain
$$
d_1(\widetilde{\alpha}_{lm}) = d_2(\widetilde{\alpha}_{lm}) = d_3(\widetilde{\alpha}_{lm}) =
d_4(\widetilde{\alpha}_{lm}) = \alpha_{lm}
$$
and we see that $\widetilde{\alpha}_{lm}$ is a lifting of $\alpha_{lm}$ from $\mathfrak{H}_3$ into
$\mathfrak{H}_4$.

To lift $\widetilde{\alpha}_{lm}$ in $\mathfrak{H}_5$ take the element
$$
\beta_{lm} = \widetilde{\alpha}_{lm} [A_{35}^l, A_{45}^m] [A_{25}^l, A_{45}^m] [A_{25}^l, A_{35}^m]
[A_{15}^l, A_{45}^m] [A_{15}^l, A_{35}^m] [A_{15}^l, A_{25}^m].
$$
We have
$$
d_5(\beta_{lm}) = d_4(\beta_{lm}) = d_3(\beta_{lm}) = d_2(\beta_{lm}) = d_1(\beta_{lm}) = \widetilde{\alpha}_{lm}
$$
and so, we constructed a lifting of $\widetilde{\alpha}_{lm}$ into $\mathfrak{H}_5$.

\begin{prop}
The group $\mathfrak{H}_4$ is generated by elements
$$
\Delta_4,~~~\widetilde{\alpha}_{lm}, l,m \in \mathbb{Z},~~~{\rm Brun}_4.
$$
\end{prop}
\begin{proof}
We know that ${\rm Brun}_3 = U_3'$ and $U_3'$ is generated by 
the commutators $\alpha_{lm};$ $l,m \in \mathbb{Z}.$
\end{proof}

\subsection{Constructing a generating set for Cohen braids $\mathfrak{H}_n$
} Let us denote by $GB_n$ the generating set of ${\rm Brun}_n,$ $n = 3, 4, \ldots.$
In the case $n=3$ we have two generating sets for ${\rm Brun}_3.$ The first set is
$$
\alpha_{lm} = [A_{13}^l, A_{23}^m], ~~l, m \in \mathbb{Z} \setminus \{0\}
$$
and the second one as in the case $n > 3$ consists of the elements
$$
[A_{13}^u, A_{23}^v],~~u, v \in U_3.
$$
As for the cases $n = 3, 4$ we have the following statement.

\begin{prop}\label{com}
The group $\mathfrak{H}_n = \mathfrak{H}_n^{B} \cap P_n$
is a subgroup of
$$
Z \times (U_3' \leftthreetimes U_4' \leftthreetimes \ldots \leftthreetimes U_n').
$$
\end{prop}
\begin{proof}
Let $\beta \in \mathfrak{H}_n.$ Then it has a form
$$
\beta = \Delta_n^{2k} (A_{13}^{k_{13}} A_{23}^{k_{23}}) \ldots (A_{1n}^{k_{1n}} A_{2n}^{k_{2n}}
\ldots A_{n-1,n}^{k_{n-1,n}}) \gamma,
$$
for some integers $k,$ $k_{ij}$ and  $\gamma \in P_n'.$ Then
$$
d_1(\beta) = \Delta_{n-1}^{2k}  A_{12}^{k_{23}} \ldots (A_{1,n-1}^{k_{2n}}
\ldots A_{n-2,n-1}^{k_{n-1,n}}) d_1(\gamma),
$$
$$
d_2(\beta) = \Delta_{n-1}^{2k}  A_{12}^{k_{13}} \ldots (A_{1,n-1}^{k_{1n}}
\ldots A_{n-2,n-1}^{k_{n-1,n}}) d_2(\gamma),
$$
$$
..........................................................
$$
$$
d_n(\beta) = \Delta_{n-1}^{2k}  (A_{13}^{k_{13}} A_{23}^{k_{23}}) \ldots (A_{1,n-1}^{k_{1,n-1}}
A_{2,n-1}^{k_{2,n-1}}
\ldots A_{n-2,n-1}^{k_{n-2,n-1}}) d_n(\gamma).
$$
Since all $d_i$ are homomorphisms $P_n'$ goes into $P_{n-1}'$ by the action of $d_i$.  It follows from the equalities
$$
d_1(\beta) = d_2(\beta) = \ldots = d_n(\beta)
$$
 that all $k_{ij}$ are equal to zero and
$$
d_1(\gamma) = d_2(\gamma) = \ldots = d_n(\gamma).
$$
Since we know that
$$
P_n' = (U_3 \leftthreetimes U_4 \leftthreetimes \ldots \leftthreetimes U_n)' =
U_3' \leftthreetimes U_4' \leftthreetimes \ldots \leftthreetimes U_n',
$$
the result follows.
\end{proof}
To find a set of generators for $\mathfrak{H}_{n+1}$ we intend to construct a lifting
from ${\rm Brun}_{n}$ to $\mathfrak{H}_{n+1}$. We know that ${\rm Brun}_{n} \subset U_n.$ Let
$w \in {\rm Brun}_{n}$. Then as a word in the free group $U_n$, $w$ contains all generators  $A_{1n},$ $A_{2n},$
$\ldots,$ $A_{n-1,n}.$ Let us denote by ${\bf n}$ the set of the first $n$ positive integers ${\bf n} = \{ 1, 2, \ldots, n \}$.  Define $n$ maps $f_i,$
$i = 1, 2, \ldots, n$
from ${\bf n}$ into ${\bf n+1}$ by the rule
$$
f_i(k) = \left\{
\begin{array}{ll}
        k & {\rm if}~ k < i,  \\
        k+1 & {\rm if}~ k \geq i,
      \end{array}
\right.
~~~i = 1, 2, \ldots, n.
$$
Let $f$ is one of $f_i$ and $w = w(A_{1n}, A_{2n}, \ldots, A_{n-1,n})$ is an element  
in $P_n$. Then let us define an action
of $f$ on $w$ by the rule
$$
w^f = w(A_{f(1),f(n)}, A_{f(2),f(n)}, \ldots, A_{f(n-1),f(n)}).
$$
We claim that for $w \in {\rm Brun}_n$ the braid
$$
\widetilde{w} = w \, w^{f_1} \, w^{f_2} \, \ldots \, w^{f_{n}}
$$
is a lifting of $w$ and lies in $\mathfrak{H}_{n+1}.$
\begin{lem}
If $w \in {\rm Brun}_n$ then the braid $\widetilde{w}$ lies in $\mathfrak{H}_{n+1}$ and $d_1(\widetilde{w}) = w.$
\end{lem}
\begin{proof}
As it was proved in \cite{BMVW}
$$
{\rm Brun}_n = [[R_1, R_2, \ldots, R_{n-1}]]_S = \prod_{\sigma \in \Sigma_{n-1}} [R_{\sigma(1)}, R_{\sigma(2)},
\ldots, R_{\sigma(n-1)}],
$$
where
$R_i = \langle\langle A_{in}\rangle\rangle^{P_n}$ is the normal closure of $A_{in}$ in $P_n.$ On the other hand,
the subgroup $U_n = \langle A_{1n}, A_{2n}, \ldots, A_{n-1,n},\rangle$ is normal in $P_n$ and from conjugation
rules we see that $R_i = \langle\langle A_{in}\rangle\rangle^{P_n} = \langle\langle A_{in}\rangle\rangle^{U_n}.$
Hence, ${\rm Brun}_n$ is generated by elements
\begin{equation}
[A_{\sigma(1),n}^{u_1}, A_{\sigma(2),n}^{u_2}, \ldots, A_{\sigma(n-1),n}^{u_{n-1}}],~~~\sigma \in \Sigma_{n-1},
\label{eq:gen}
\end{equation}
where $u_i = u_i(A_{1n}, A_{2n}, \ldots, A_{n-1,n})$ are elements in $U_n.$

Take a generator
$$
w = [A_{\sigma(1),n}^{u_1}, A_{\sigma(2),n}^{u_2}, \ldots, A_{\sigma(n-1),n}^{u_{n-1}}]
$$
of ${\rm Brun}_n$. For simplicity we will assume that the permutation $\sigma$ is trivial. The proof in 
general case is similar. Let us construct a braid
$$
\widetilde{w} = w^{f_0} \, w^{f_1} \, w^{f_2} \, \ldots \, w^{f_{n}},
$$
which lies in $P_{n+1}.$ To prove that this braid is a Cohen braid, find
\begin{multline*}
\widetilde{w} = [A_{1n}^{u_1}, A_{2n}^{u_2}, \ldots, A_{n-1,n}^{u_{n-1}}] \,
[A_{2,n+1}^{u_1^{f_1}}, A_{3,n+1}^{u_2^{f_1}}, \ldots, A_{n,n+1}^{u_{n-1}^{f_1}}] \cdot \\
[A_{1,n+1}^{u_1^{f_2}}, A_{3,n+1}^{u_2^{f_2}}, \ldots, A_{n,n+1}^{u_{n-1}^{f_2}}] \ldots
[A_{1,n+1}^{u_1^{f_n}}, A_{2,n+1}^{u_2^{f_n}}, \ldots, A_{n,n+1}^{u_{n-1}^{f_n}}].
\end{multline*}
Then
$$
d_i(\widetilde{w}) = [A_{1n}^{d_i(u_1^{f_i})}, A_{2n}^{d_i(u_2^{f_i})}, \ldots, A_{n-1,n}^{d_i(u_{n-1}^{f_i})}],
~~~i = 1, 2, \ldots, n.
$$
We have:
\begin{multline*}
d_i(u_j^{f_i}) = d_i(u_j(A_{1,n+1}, \ldots, A_{i-1,n+1}, A_{i+1,n+1}, \ldots, A_{n,n+1})) = \\
u_j(A_{1n}, \ldots, A_{i-1,n}, A_{i,n}, \ldots, A_{n-1,n}) = u_j,~~~j = 1, 2, \ldots, n-1.
\end{multline*}
Hence,
$$
d_i(\widetilde{w}) = [A_{1n}^{u_1}, A_{2n}^{u_2}, \ldots, A_{n-1,n}^{u_{n-1}}] = w ~\mbox{for all}~i = 1, 2, \ldots,
n.
$$
On the other hand
$$
d_{n+1}(\widetilde{w}) = w
$$
and we see, that $\widetilde{w}$ lies in $\mathfrak{H}_{n+1}.$

If $w = w_1 w_2 \ldots w_k \in {\rm Brun}_n$ is a product of generators of the type (\ref{eq:gen}) or their inverses,  then
$$
\widetilde{w} = w_1 w_2 \ldots w_k (w_1 w_2 \ldots w_k)^{f_1} \ldots (w_1 w_2 \ldots w_k)^{f_n}
$$
and as in the previous case we can check that $d_i(\widetilde{w}) = w$ for all $i = 1, 2, \ldots, n+1.$
\end{proof}
Let us consider the following

\noindent
{\bf Question.} Let $w = w_1 w_2 \ldots w_k \in {\rm Brun}_n$ is the product of generators and their inverses
from $GB_n$. We can find $\widetilde{w_j}$ for all $j$ and take their product. We will have
$$
\widetilde{w_1} \widetilde{w_2} \ldots \widetilde{w_k} = (w_1 w_1^{f_1} \ldots w_1^{f_n})
(w_2 w_2^{f_1} \ldots w_2^{f_n})
\ldots (w_k w_k^{f_1} \ldots w_k^{f_n}).
$$
Is it true that
$$
\widetilde{w} = \widetilde{w_1} \widetilde{w_2} \ldots \widetilde{w_k} u
$$
for some Brunnian braid $u \in {\rm Brun}_{n+1}?$

\medskip

We know that ${\rm Brun}_m$ is generated by the commutators
$$
w = [A_{i_1,m}^{u_1}, A_{i_2,m}^{u_2}, \ldots, A_{i_{m-1},m}^{u_{m-1}}],~~\mbox{where all}~u_i \in U_m
$$
and $\{ i_1, i_2, \ldots, i_{m-1} \} = {\bf m-1}.$

To generalize the previous construction and to construct a lifting of ${\rm Brun}_m,$ 
$3 \leq m < n+1$
to
$\mathfrak{H}_{n+1}$
let us define a map
$$
T_{m,n+1} : {\rm Brun}_m \longrightarrow \mathfrak{H}_{n+1}
$$
by the formula
$$
T_{m,n+1}(\alpha) = \prod_{k=m}^{n+1} \tau_{m,k}(\alpha),~~\alpha \in {\rm Brun}_m,
$$
where
$$
\tau_{m,k} : {\rm Brun}_m \longrightarrow U_k'
$$
and
$$
\tau_{m,m}(\alpha) = \alpha,~~~
$$
$$
\tau_{m,k}(\alpha) = \prod_{1 \leq i_1 < i_2 < \ldots < i_{k-m} \leq k-1}
\alpha^{f_{i_1} f_{i_2}\ldots f_{i_{k-m}}},~~
k > m,
$$
with the lexicographic order from the right.
In particular, in this language, the lifting
$$
\widetilde{ } : {\rm Brun}_n \longrightarrow \mathfrak{H}_{n+1}
$$
that was constructed above has the form
$$
\widetilde{w} = \tau_{n,n}(w) \tau_{n,n+1}(w) = w \tau_{n,n+1}.
$$

\noindent
{\bf Examples.} 
\smallskip

\noindent
1) The map
$$
\tau_{3,4} : {\rm Brun}_3 \longrightarrow U_4'
$$
is defined by the rule
$$
\tau_{3,4}(\alpha) = \alpha^{f_1} \alpha^{f_2} \alpha^{f_3},~~~\alpha \in {\rm Brun}_3.
$$

\noindent
2) The map
$$
\tau_{3,5} : {\rm Brun}_3 \longrightarrow U_5'
$$
is defined by the rule
$$
\tau_{3,5}(\alpha) = \alpha^{f_1 f_2} \alpha^{f_1 f_3} \alpha^{f_1 f_4}
\alpha^{f_2 f_3} \alpha^{f_2 f_4} \alpha^{f_3 f_4},~~~\alpha \in {\rm Brun}_3.
$$

These constructions give the proof of the main result of this section.

\begin{thm}
The group $\mathfrak{H}_n,$ $n \geq 3$, is generated by the set
$$
\Delta_n, T_{3,n}(GB_3), T_{4,n}(GB_4), \ldots, GB_n.
 $$ \hfill $\Box$
\end{thm}
{\bf Remarks.} 1) The action of the maps $f_i$ on the generators $A_{kl}$ is the same as the actions
of the homomorphisms $d^i : B_n \longrightarrow B_{n+1}$ that was defined in the beginning of section 2.

2) The lifting $T_{m,n+1}$ is the similar to the James-Hopf operation $H_{m,n+1}$
(see subsection~\ref{JHm}).  The main difference is
that we consider the normal form of the element $w \in P_n$ of the type $w = u_2 u_3 \ldots u_n,$ $u_i \in U_i$,
but in the case of the James-Hopf operation one uses the normal form of the type $w = u_n u_{n-1} \ldots u_2,$
$u_i \in U_i$.

\section{Cohen braids  of the sphere and of the projective plane}
\subsection{Cohen braids of the sphere}
For braids on three strands we  have: $\mathfrak{H}_3(S^2)= \Brun_3(S^2)=P_3(S^2)=\Z/2$, the map
$\mathfrak{H}_4(S^2)\buildrel d_1\over\longrightarrow \mathfrak{H}_{3}(S^2)$
is obviously onto and $\Delta_3^2$ is central in  $\mathfrak{H}_4(S^2)$, so we have
$$\mathfrak{H}_4(S^2)\cong\Z/2\times\Brun_4(S^2).$$
\subsection{Cohen braids of the projective plane}
Let us consider the case of $\RP^2$ for very small number of strands.
Group $B_n(\mathbb{R}P^2)$ is generated by elements
$$
\sigma_1, \sigma_2, \ldots, \sigma_{n-1}, \rho
$$
and is defined by the braid relations of $B_n$ among $
\sigma_1, \sigma_2, \ldots, \sigma_{n-1}, $ and by the following relations: \\

$\rho \sigma_i = \sigma_{i} \rho,$ $i \not= 1;$\\

$\sigma_1^{-1} \rho \sigma_1^{-1} \rho = \rho \sigma_{1}^{-1} \rho \sigma_1;$\\

$\rho^2 =  \sigma_1 \sigma_{2} \ldots \sigma_{n-2} \sigma_{n-1}^2 \sigma_{n-2} \ldots \sigma_{2} \sigma_1.$\\

\noindent
There exists the canonical homomorphism
$$
\tau : B_n(\mathbb{R}P^2) \longrightarrow \Sigma_n,
$$
where $\tau(\sigma_i) = (i,i+1),$ $\tau(\rho) = e.$ The kernel of $\tau$ is the 
pure braid group
$P_n(\mathbb{R}P^2).$
If $n=1$ then $B_1(\mathbb{R}P^2) = P_1(\mathbb{R}P^2) = \pi_1(\mathbb{R}P^2) = \mathbb{Z} / 2.$
If $n=2$ then
$$
B_2(\mathbb{R}P^2) = \langle \sigma_{1}, \rho ~||~ \sigma_1^{-1} \rho 
\sigma_1^{-1} \rho =
\rho \sigma_{1}^{-1} \rho \sigma_1,   \rho^2 = \sigma_1^2 \rangle
$$
has order 16 and $P_2(\mathbb{R}P^2)$ is the  quaternion group of order 8, which has a presentation
$$
P_2(\mathbb{R}P^2) = \langle u, \rho ~||~  \rho u \rho^{-1} = u^{-1},   \rho^2 = u^2 \rangle,
$$
where $u = \sigma_1 \rho \sigma_1^{-1}.$

\noindent
If $n=3$ then
\begin{multline*}
B_3(\mathbb{R}P^2) = 
\langle \sigma_{1},  \sigma_{2}, \rho ~||~ \sigma_1 \sigma_2 \sigma_1  = \sigma_2 \sigma_1
 \sigma_2, ~~\rho \sigma_2 = \sigma_2 \rho,~~\\
\sigma_1^{-1} \rho \sigma_1^{-1} \rho =
\rho \sigma_{1}^{-1} \rho \sigma_1,  ~~ \rho^2 = \sigma_1 \sigma_2^2 \sigma_1 \rangle
\end{multline*}
and $P_3(\mathbb{R}P^2)$ is generated by elements
$$
\rho, u = \sigma_1 \rho \sigma_1^{-1}, w = \sigma_2 \sigma_1 \rho \sigma_1^{-1} \sigma_2^{-1},
A_{12} = \sigma_1^2, A_{23} = \sigma_2^2,  A_{13} = \sigma_2 \sigma_1^2 \sigma_2^{-1}
$$
(see \cite{BMVW}).
The subgroup
$$
U_3(\mathbb{R}P^2) = \langle w, A_{13}, A_{23} ~||~  A_{13} A_{23} = w^2 \rangle
$$
is normal in $P_3(\mathbb{R}P^2)$, free of rank 2 and
$$
P_3(\mathbb{R}P^2) / U_3(\mathbb{R}P^2) \cong P_2(\mathbb{R}P^2).
$$
It was proved in 
\cite[Subsection~3.1]{BMVW} that $\Brun_2(\RP^2)$ is the normal closure of the element 
$A_{1,2}$ in $B_2(\RP^2)$.
Consider the commutative diagram of fiber sequences \cite[page 1618]{BMVW}
\begin{equation} 
\begin{diagram}
\RP^2\smallsetminus\{p_1\}&\rInto^{i_2}&F(\RP^2,2)&\rTo^{\delta^{(2)}}&\RP^2\\
\dInto>{i}&&\dTo>{\delta^{(1)}}&&\dTo\\
\RP^2&\rEq&\RP^2&\rTo&\ast,\\
\end{diagram}
\label{drp2}
\end{equation}
where $i_2(x)=(p_1,x)$ and $i_1(x)=(x,p_2)$.
From the upper row, there is an exact sequence
\begin{multline}
\pi_2(\RP^2)\rTo \pi_1(\RP^2\smallsetminus \{p_1\})\rTo^{i_{2*}} \pi_1(F(\RP^2,2))=\\
P_2(\RP^2)\rOnto^{d_2}
 \pi_1(\RP^2)=P_1(\RP^2)=\Z/2.
 \label{ex_se_pi}
\end{multline}
As the group $P_2(\RP^2)$ is the quaternion group of order 8, the first map in (\ref{ex_se_pi})
is the multiplication by 4.
Note that
\begin{multline*}
\Brun_2(\RP^2)=\\
\Ker(d_1\colon P_2(\RP^2)\to P_1(\RP^2))\cap \Ker(d_2\colon P_2(\RP^2)\to P_1(\RP^2))=\\
\Ker(d_1\colon \Ker(d_2\colon P_2(\RP^2)\to P_1(\RP^2))\to P_1(\RP^2)).
\end{multline*}
We see from diagram (\ref{drp2}) that the map 
\begin{equation*}
d_1\colon \Ker(d_2\colon P_2(\RP^2)\to P_1(\RP^2))\to P_1(\RP^2)
\end{equation*}
is onto, so  
$\Brun_2(\RP^2)$ has order 2.
Consider the following commutative diagram
\begin{equation}
\begin{diagram}
\Brun_{2}(\RP^2)&\rInto &\mathfrak{H}_{2}(\RP^2)&\rTo^{d_1}&\Z/2\\
\dTo&&\dTo&&\dTo\,\Delta\\
\Brun_{2}(\RP^2)&\rInto &P_2(\RP^2)&\rTo^{(d_1, d_2)}&\Z/2\times\Z/2,\\
\end{diagram}
\label{d2rp2}
\end{equation}
As $\Brun_2(\RP^2)$ has order 2, the map $(d_1,d_2)$ should be onto.
Then by definition of $\mathfrak{H}_{2}(\RP^2)$ the map $d_1$ in (\ref{d2rp2})
should be onto.
Consider the exact sequence
\begin{equation*}
1\to\Brun_2(\RP^2)\to \mathfrak{H}_2(\RP^2)\buildrel{d_1}\over\longrightarrow
 P_1(\RP^2)\to 1.
\end{equation*} 
It follows from this sequence that $ \mathfrak{H}_2(\RP^2)$ has order 4 and as it 
is a subgroup of 
the quaternion group it should be $\Z/4$. 
\begin{lem}
The element $\rho u w$ is a Cohen in $P_3(\mathbb{R}P^2).$
\end{lem}
\begin{proof}
Indeed,
$$
d_1(\rho u w) = d_2(\rho u w) = d_3(\rho u w) = \rho u.
$$
\end{proof}
It was proved in \cite{BMVW} that ${\rm Brun}_3(\mathbb{R}P^2) \subset U_3(\mathbb{R}P^2)$ and it is a free subgroup
of finite rank with the set of generators
$$
w^4,~~A_{23}^2,~~[w^4, A_{23}],~~[A_{23}, w],~~[A_{23}, w^2],~~[A_{23}, w^3],
$$
$$
[[A_{23}, w], A_{23}],~~[[A_{23}, w^2], A_{23}]~~[A_{23}, w^3], A_{23}].
$$


\begin{prop}
Group $\mathfrak{H}_2(\mathbb{R}P^2)$ (as a cyclic of order 4) has a generator $\rho u.$
\end{prop}
\begin{proof}
In $P_2(\mathbb{R}P^2)$ element $u$ has order 4 and every element has the form
$$
\rho^{\varepsilon} u^{\mu},~~~\varepsilon = 0, 1,~~\mu = 0, 1, 2, 3, 4.
$$
Let $\varepsilon = 0$, we can check that $u$ and $u^3$ are not Cohen but $u^2 = \rho^2$ is.
Let $\varepsilon = 1.$ Since $\rho u$ is Cohen then the elements
$$
(\rho u)^2 = u^2,~~(\rho u)^3 = \rho u^3
$$
are Cohen. The element $\rho u^2$ is not Cohen.
\end{proof}
To study the group $\mathfrak{H}_3(\mathbb{R}P^2)$ let us construct a lifting of
 elements of
$\mathfrak{H}_2(\mathbb{R}P^2)$ to $\mathfrak{H}_3(\mathbb{R}P^2).$ A lifting of $\rho u$ is equal
to $\rho u w.$ However, the elements $\rho$ and $u$ in $P_3(\mathbb{R}P^2)$ are not 
the same as elements
$\rho$ and $u$ in $P_2(\mathbb{R}P^2).$
To simplify  the calculations let us define the elements $a = \rho w$ and $b = w u.$ 
Then the group
$\langle a, b\rangle$ is isomorphic to $P_2(\mathbb{R}P^2)$  \cite{BMVW}. 
We find the powers of $\rho u w$:

$\rho u w = a b (w^{-1} A_{23})^2 w,$

$(\rho u w)^2 =  b^2  A_{23}^{-1} w A_{23}^{-1} (w^{-1} A_{23})^2 w,$

$(\rho u w)^3 = a b^3  (w^{-1} A_{23})^{2} (w A_{23}^{-1})^2 (w^{-1} A_{23})^2 w,$

$(\rho u w)^4 = A_{23}^{-1} w A_{23}^{-1} (w^{-1} A_{23})^{2} (w A_{23}^{-1})^2 (w^{-1} A_{23})^2 w,$

$.......................................................................................$
\medskip

\noindent
In particular, the element $\rho u w$ has infinite order in $P_3(\mathbb{R}P^2)$ but the element $\rho u$ has order 4 in $P_2(\mathbb{R}P^2)$.
\begin{lem}
The group $\mathfrak{H}_3(\mathbb{R}P^2)$ is generated by $\rho u w$ and the elements
$$
w^4,~~A_{23}^2,~~[w^4, A_{23}],~~[A_{23}, w],~~[A_{23}, w^2],~~[A_{23}, w^3],
$$
$$
[[A_{23}, w], A_{23}],~~[[A_{23}, w^2], A_{23}]~~[A_{23}, w^3], A_{23}].
$$
In particular, $\mathfrak{H}_3(\mathbb{R}P^2)$ is finitely generated.
\hfill $\Box$
\end{lem}
Note that the similar fact in $B_3$ does not hold. Indeed, as we know
$\mathfrak{H}_3 = Z(B_3) \times U_3'$ and $U_3'$ is not finitely generated.

\medskip
\noindent 
{\bf Question}. Whether $\mathfrak{H}_4(\mathbb{R}P^2)$ is finitely generated or not?

\subsection{Map from Cohen braids of the disc to that of the sphere}\label{JHm}
This study is motivated by the facts that for the braids and pure braids these maps are epimorphisms and for the Brunnian braids the following exact sequence was proved in \cite{BCWW}
\begin{equation*}
1\to \Brun_{n+1}(S^2) \to \Brun_{n}(D^2) \to \Brun_{n}(S^2) \to 
\pi_{n-1}(S^2)\to 1.
\label{eq:maks2}
\end{equation*} 
Let $M$ be a connected $2$-manifold with nonempty boundary. For $n\geq k$, the \textit{James-Hopf operation} is defined as a map $$H_{k,n}\colon \Brun_k(M)\rTo \mathfrak{H}_n(M)$$
by setting $H_{k,k}(\beta)=\beta$ with
\begin{equation*}\label{equation6.4}
H_{k,n}(\beta)=\prod_{1\leq i_1<i_2<\cdots<i_{n-k}\leq n}d^{i_{n-k}}d^{i_{n-k-1}}\cdots d^{i_1}(\beta)
\end{equation*}
with lexicographic order from right for $\beta\in \Brun_k(M)$. The map $H_{k,n}$ satisfies the property that
$$
d_iH_{k,n}(\beta) =H_{k,n-1}(\beta)
$$
for $1\leq i\leq n$, $k< n$ and $\beta\in \Brun_k(M)$. Thus the sequence of elements $\{H_{k,n}(\beta)\}$ lifts to $\mathfrak{H}(M)$ that defines a map
$$
H_{k,\infty}\colon \Brun_k(M)\rTo \mathfrak{H}(M).
$$
According to~\cite[Theorem 3.4]{Wu4}, we have the following proposition.
\begin{prop}\label{proposition6.5}
Let $M$ be a connected $2$-manifold with nonempty boundary and let $\alpha\in\mathfrak{H}_n(M)$ with $1\leq n\leq \infty$. Then there exists an unique element $\delta_k(\alpha)\in \Brun_k(M)$ for $1\leq k\leq n$ such that the equality
    $$
    \alpha=\prod_{k=1}^n H_{k,n}(\delta_k(\alpha))
    $$
holds.\hfill $\Box$
\end{prop}
The coface homomorphism $d^i\colon B_n\to B_{n+1}$ is given by the formula \cite[Example 1.2.8]{Wu4}
\begin{equation*}\label{eq:3.1}
d^i\sigma_j=\left\{
\begin{array}{lcl}
\sigma_j&\textrm{ for }& j<i-1,\\
\sigma_i\sigma_{i-1}\sigma_i^{-1}& \textrm{ for } & j=i-1,\\
\sigma_{j+1}&\textrm{ for } &j>i-1.\\
\end{array}\right.
\end{equation*}
Remind that the generators for the pure braid group $P_n$ are given by
\begin{equation*}\label{eq:3.2}
A_{i,j}=\sigma_{j-1}\sigma_{j-2}\cdots\sigma_{i+1}\sigma_i^2\sigma_{i+1}^{-1}\cdots
\sigma_{j-2}^{-1}\sigma_{j-1}^{-1}
\end{equation*}
for $1\leq i<j\leq n$. Combining the above two formulae, we have 
\begin{equation*}\label{eq:3.3}
d^iA_{s,t}=\left\{
\begin{array}{lcl}
A_{s+1,t+1}&\textrm{ for }& i\leq s,\\
A_{s,t+1}&\textrm{ for }& s+1\leq i\leq t,\\
A_{s,t}&\textrm{ for } & i>t.\\
\end{array}\right.
\end{equation*}

\medskip

\noindent\textbf{Example.} Consider the James-Hopf operation 
$$H_{2,4}\colon \Brun_2(D^2)\to \mathfrak{H}_4(D^2).$$ 
From the definition,
$$
\begin{array}{cl}
&H_{2,4}(A_{1,2})=\\
=&d^2d^1(A_{1,2})d^3d^1(A_{1,2})d^3d^2(A_{1,2})d^4d^1(A_{1,2})
d^4d^2(A_{1,2})d^4d^3(A_{1,2})\\
=&d^2(A_{2,3})d^3(A_{2,3})d^3(A_{1,3})d^4(A_{2,3})d^4(A_{1,3})d^4(A_{1,2})\\
=&A_{3,4}A_{2,4}A_{1,4}A_{2,3}A_{1,3}A_{1,2}.\\
\end{array}
$$
Similarly $H_{2,4}(A_{1,2}^k)=A_{3,4}^kA_{2,4}^kA_{1,4}^k A_{2,3}^kA_{1,3}^kA_{1,2}^k$ for $k\in \Z$. 

\section{Generalizations of  Brunian and Cohen braids} 



Let $\beta$ be an element in the  braid group
$B_n(M)$, and let $\bf{n_1, n_2, \ldots,n_k}$ be non-empty subsets of ${\bf n} = \{ 1, 2, \ldots, n \}$ such that
${\bf n_i} \cap {\bf n_j} = \emptyset$ if $i \not= j.$ We will say that $\beta$ is a
$({\bf n_1},{\bf n_2}, \ldots {\bf n_k}$--{\it Cohen} braid if
$d_{i_1}(\beta) = d_{i_2}(\beta) = \ldots = d_{i_s}(\beta),$
where ${\bf n_i} = \{ i_1, i_2, \ldots, i_s \}$ for all $i = 1, 2, \ldots, k.$ In particular, ${\bf n}$--Cohen
braid in $B_n(M)$
is an element of $\mathfrak{H}^B_n(M)$.

If ${\bf m} = \{ i_1, i_2, \ldots, i_s \} \subseteq {\bf n}$ then ${\bf m}$--Cohen braid $\beta$ is called ${\bf m}$--{\it Brumnnian} braid
if $d_{i_1}(\beta) = d_{i_2}(\beta) = \ldots = d_{i_s}(\beta) = e \in B_{n-1}(M).$

If $\beta$ is a $({\bf n_1},{\bf n_2}, \ldots,{\bf n_k})$--Cohen braid then we denote
$d_{\bf n_i}(\beta) = d_{i_1}(\beta) = d_{i_2}(\beta) = \ldots = d_{i_s}(\beta),$ where
${\bf n_i} = \{ i_1, i_2, \ldots, i_s \}$ for all $i = 1, 2, \ldots, k.$  
Let $\alpha_1, \alpha_2, \ldots, \alpha_k$ be some braids in $B_{n-1}(M)$ and
$\bf{n_1, n_2, \ldots,n_k}$
be non-empty subsets of ${\bf n} = \{ 1, 2, \ldots, n \}$ such that
${\bf n_i} \cap {\bf n_j} = \emptyset$ if $i \not= j.$
We ask the following question.
Does there exists a braid $\beta \in B_n(M)$ such that it is a solution of the following system of equations
\begin{equation}
\begin{cases}
 d_{\bf n_1}(\beta)=\alpha_1, \\
\dots\\
 d_{\bf n_k}(\beta)=\alpha_k.
 \end{cases}\label{eq:sys4}
\end{equation}
This generalises the system of equations (\ref{eq:sys2}).

One of the motivations for the definition of ${\bf m}$--Brunnian braids comes from the work of Makanin \cite{M}. He
defined the set of unary braids in $B_n$ and the set of harmonic braids in $P_n$. A braid $\beta \in B_n$
is called {\it unary} if its string starting in $1$ ends in $n$ and if take away this string then the other $n-1$ strings of $\beta$ form
the trivial braid. In particular, an unary braid is $\{ 1 \}$--Brunnian braid by our definition. The closure of unary braid is a knot.
Makanin proved that passing to the closure establishes a bijection between isotopy classes of oriented knots and equivalence classes of unary braids generated by the Markov moves. A pure braid $\gamma \in P_n$ is called
{\it harmonic} if  its Markov normal form $\gamma = \gamma_2 \gamma_3 \ldots \gamma_n,$ $\gamma_i \in U_i,$ satisfies the  following property:
$$
d_1(\gamma_i) = \gamma_{i-1}
$$
for all  $i = 3, 4, \ldots, n$.
Harmonic braids form a subgroup $H_n$ of $P_n.$ The set of unary braids on $n$ strings is invariant under
conjugation by $H_n$ and unary braids conjugate in $B_n$ are conjugate by an element of $H_n$.
\begin{prop}
Any unary braid $\beta$ in $B_n$ has a form 
$$\beta = \beta_0 \sigma_1 \sigma_2 \ldots
 \sigma_{n-1},$$
where $\beta_0 \in \langle\langle A_{12}, A_{13}, \ldots, A_{1n}\rangle\rangle^{P_n}.$
\hfill $\Box$
\end{prop}
In general, the set of $({\bf n_1},{\bf n_2}, \ldots, {\bf n_k})$--
Cohen braids is not a subgroup of $B_n(M)$. Indeed, in $B_3$ the braids $\sigma_1$ and $\sigma_2 \sigma_1$ are
$\{1, 2 \}$--Cohen braids but the element $\sigma_1^{-1} \sigma_2^{-1}$ that is inverse to $\sigma_2 \sigma_1$
is not $\{1, 2 \}$--Cohen. However, if we consider 
$({\bf n_1},{\bf n_2}, \ldots, {\bf n_k})$--Cohen braids in the pure braid group
then there this set forms a subgroup. Denote it by
$\mathfrak{H}_{({\bf n_1},{\bf n_2},\ldots, {\bf n_k})}(M).$

\medskip
{\bf Question}. Find a generating set for $\mathfrak{H}_{({\bf n_1},{\bf n_2}, \ldots, {\bf n_k})}(M).$
 
\section{Some more open questions}

1. Let $\alpha$ be an $n$-component Brunnian link, then all finite type invariants of order
$\leq n-1$ of $\alpha$ vanish. In the case when $\alpha$
is the plat-closure of a pure braid this was shown by Kalfagianni and Lin \cite{KL} . What can we say about other invariants of $\alpha$? In particular, about the group
$\pi_1(S^3 \setminus \alpha),$ about the Alexander polynomial of $\alpha,$ about the Jones polynomial of $\alpha$?
\smallskip

2. Link groups  are residually finite \cite[Theorem 6.3.9]{K}. The question of residual nilpotency
of link groups is more complicated. In particular, the group of any non-trivial
knot is not residually nilpotent,
but the groups of some links are. It was proved in \cite{BM}  that the group of Whitehed link and group of Borromean
link are residually nilpotent, on the other side there exists a 2-component 2-bridge Brunnian link whose group
is not residually nilpotent. So, the question is:

\noindent
 What groups of Brunnian links are residually (torsion free) nilpotent?
\smallskip 

3. The existence of the following exact sequence
\begin{equation*}
1\to\Brun_n(M)\to \mathfrak{H}_n(M)\buildrel d_1\over\longrightarrow \mathfrak{H}_{n-1}(M)\to 1
\end{equation*}
was proved in Corollary~\ref{corollary6.5}.
Is there a cross-section $s : \mathfrak{H}_{n-1}(M) \longrightarrow \mathfrak{H}_n(M),$ (in what cases?) i.e. such an embedding of
$\mathfrak{H}_{n-1}(M)$ into $\mathfrak{H}_{n}(M)$ that
$\mathfrak{H}_{n}(M) = {\rm Brun}_n(M) \leftthreetimes \mathfrak{H}_{n-1}(M)$?
\smallskip

4. Garside proved that any braid $\beta \in B_n$ has a normal form
$\beta = \Delta_n^k \beta_0,$ where $k$ is an integer and $\beta_0 \in B_n^+$ is a positive braid.

\smallskip
{\bf Question.} What positive braids are Cohen?

\smallskip
Braids $\Delta_n^k$ lie in the intersection $B_n^+ \cap \mathfrak{H}_n^B$ for all $k > 0$. Also, in this
intersection lie the braids
$$
A_{12}^k (A_{13}^k A_{23}^k) \ldots (A_{1n}^k A_{2n}^k \ldots A_{n-1,n}^k),~~~k > 0.
$$
On the other side, if $\beta \in \mathfrak{H}_n^B$ then we can find its Garside normal form
$\beta = \Delta_n^k \beta_0$ and $\beta_0 \in B_n^+ \cap \mathfrak{H}_n^B.$ If we know
 that a positive
braid $\beta_0$ is Cohen braid then any braid of the form $\Delta_n^k \beta_0$ is also
 Cohen for any integer $k$.
\section{Acknowledgments}
The last author is partially supported by the Singapore Ministry
of Education research grant (AcRF Tier 2 WBS No. R-146-000-143-112)
and a grant (No. 11028104) of NSFC of China.

\end{document}